\documentclass[a4paper,10pt]{article}
\usepackage{amsfonts,amsthm,amsmath}
\usepackage{graphicx}
\usepackage{geometry}
\usepackage{longtable}

\def\tAn{\tilde{A}_n}
\def\mQ{\mathcal{Q}}
\def\mH{\mathcal{H}}
\def\mC{\mathcal{C}}
\def\Hom{\mathrm{Hom}}
\def\add{\mathrm{add}}
\def\End{\mathrm{End}}
\def\Z{\mathbb{Z}}
\def\rn#1{\romannumeral#1}
\def\N{\mathbb{N} }

\theoremstyle{plain}
\newtheorem{thm}{Theorem}[section]
\newtheorem{prop}[thm]{Proposition}
\newtheorem{lem}[thm]{Lemma}
\newtheorem{cor}[thm]{Corollary}

\theoremstyle{definition}
\newtheorem{dfn}[thm]{Definition}
\newtheorem{eg}[thm]{Example}
\newtheorem{nt}[thm]{Notation}
\newtheorem{rmk}[thm]{Remark}
\newtheorem{alg}[thm]{Algorithm}

\title{Mutation classes of $\tilde{A}_n-$quivers and derived equivalence
classification of cluster tilted algebras of type $\tilde{A}_n$}
\author{Janine Bastian}
\date{Leibniz Universit\"at Hannover,\\
Institut f\"ur Algebra, Zahlentheorie
und Diskrete Mathematik,\\
Welfengarten 1, D-30167 Hannover, Germany\\
bastian@math.uni-hannover.de}

\begin{document}

\maketitle

\begin{abstract}
We give an explicit description of the mutation classes of quivers of type $\tAn$.
Furthermore, we provide a complete classification of cluster tilted algebras of type
$\tilde{A}_n$ up to derived equivalence. We show that the bounded derived category
of such an algebra depends on four combinatorial parameters of the corresponding quiver.
\end{abstract}




\section{Introduction}

A few years ago, Fomin and Zelevinsky introduced the concept of cluster algebras
\cite{Fomin-Zelevinsky} which rapidly became a successful research area.
Cluster algebras nowadays link various areas of mathematics, like combinatorics, Lie theory,
algebraic geometry, representation theory, integrable systems, Teichm\"uller theory, Poisson
geometry and also string theory in physics (via recent work on quivers with superpotentials
\cite{DMZ}, \cite{Labardini}).

In an attempt to 'categorify' cluster algebras, which a priori are combinatorially defined,
cluster categories have been introduced by Buan, Marsh, Reineke, Reiten and Todorov \cite{BMRRT}.
For a quiver $Q$ without loops and oriented 2-cycles and the corresponding path algebra $KQ$
(over an algebraically closed field $K$), the cluster category $\mC_Q$ is the orbit category
of the bounded derived category $D^b(KQ)$ by the functor $\tau^{-1}[1]$, where $\tau$ denotes
the Auslander-Reiten translation and $[1]$ is the shift functor on the triangulated category $D^b(KQ)$.

Important objects in cluster categories are the cluster-tilting objects.
The endomorphism algebras of such objects in the cluster category $\mC_Q$ are called
cluster tilted algebras of type $Q$ \cite{BMR}.
Cluster tilted algebras have several interesting properties, e.g. their representation theory
can be completely understood in terms of the representation theory of the corresponding path
algebra of a quiver (see \cite{BMR}). These algebras have been studied by various authors,
see for instance \cite{ABS1}, \cite{ABS2}, \cite{BMR2} or \cite{CCS}.

\smallskip

In recent years, a focal point in the representation theory of algebras has been the
investigation of derived equivalences of algebras. Since a lot of properties and invariants
of rings and algebras are preserved by derived equivalences, it is important for many purposes
to classify classes of algebras up to derived equivalence, instead of Morita equivalence.
For selfinjective algebras the representation type is preserved under derived equivalences
(see \cite{Krause} and \cite{Rickard3}).
It has been also proved in \cite{Rickard2} that the class of symmetric algebras is closed under
derived equivalences. Additionally, we note that derived equivalent algebras have the same number
of pairwise nonisomorphic simple modules and isomorphic centers.

\smallskip

In this work, we are concerned with the problem of derived equivalence classification of cluster
tilted algebras of type $\tAn$. Such a classification was done for cluster tilted algebras of
type $A_n$ by Buan and Vatne in $2007$ \cite{Buan-Vatne}; see also the work of Murphy on the more
general case of $m$-cluster tilted algebras of type $A_n$ \cite{Murphy}.

Since the quivers of cluster tilted algebras of type $\tAn$ are exactly the quivers in the
mutation classes of $\tAn$, our first aim in this paper is to give a description of the mutation
classes of $\tAn-$quivers; these mutation classes are known to be finite (for example see \cite{FST}).
The second purpose of this note is to describe, when two cluster tilted algebras of type $Q$
have equivalent derived categories, where $Q$ is a quiver whose underlying graph is $\tAn$.

\smallskip

In Definition \ref{description_mutation_class} we present a class $\mQ_n$ of quivers with $n+1$
vertices which includes all non-oriented cycles of length $n+1$. To show that this class contains
all quivers mutation equivalent to some quiver of type $\tAn$ we first prove that this class is
closed under quiver mutation.
Furthermore, we define parameters $r_1, \, r_2, \ s_1$ and $s_2$ for any quiver $Q \in \mQ_n$
in Definition \ref{parameters} and prove that every quiver in $\mQ_n$ with parameters
$r_1, \, r_2, \ s_1$ and $s_2$ can be mutated to a normal form, see Figure~\ref{normal_form},
without changing the parameters.

With the help of the above result we can show that every quiver $Q \in \mQ_n$ with parameters
$r_1, \, r_2, \ s_1$ and $s_2$ is mutation equivalent to some non-oriented cycle with $r:=r_1+2r_2$
arrows in one direction and $s:=s_1+2s_2$ arrows in the other direction.
Hence, if two quivers $Q_1$ and $Q_2$ of $\mQ_n$ have the parameters $r_1, \ r_2, \ s_1, \ s_2$,
respectively $\tilde{r}_1, \ \tilde{r}_2, \ \tilde{s}_1, \ \tilde{s}_2$ and
$r_1+2r_2 = \tilde{r}_1 + 2\tilde{r}_2, \ s_1+2s_2 = \tilde{s}_1 + 2\tilde{s}_2$ (or vice versa),
then $Q_1$ is mutation equivalent to $Q_2$.

The converse of this result, i.e., an explicit description of the mutation classes of quivers
of type $\tAn$, can be shown with the help of Lemma $6.8$ in \cite{FST}.

\smallskip

The main result of the derived equivalence classification of cluster tilted algebras of type
$\tAn$ is the following theorem:

\begin{thm}
Two cluster tilted algebras of type $\tAn$ are derived equivalent if and only if their quivers
have the same parameters $r_1, \ r_2, \ s_1$ and $s_2$ (up to changing the roles of $r_i$ and
$s_i$, $i\in \{1,2\}$).
\end{thm}

We prove that every cluster tilted algebra of type $\tAn$ with parameters $r_1, \, r_2, \ s_1$
and $s_2$ is derived equivalent to a cluster tilted algebra corresponding to a quiver in normal form.
Furthermore, we compute the parameters $r_1, \ r_2, \ s_1$ and $s_2$ as combinatorial derived
invariants for a quiver $Q \in \mQ_n$ with the help of an algorithm defined by Avella-Alaminos
and Gei\ss \ in \cite{Avella-Geiss}.

\medskip

The paper is organized as follows. In Section 2 we collect some basic notions about quiver
mutations.
In Section 3 we present the set $\mQ_n$ of quivers which can be obtained by iterated mutation
from quivers whose underlying graph is of type $\tAn$.
Moreover, we describe, when two quivers of $\mQ_n$ are in the same mutation class.
In the fourth section we describe the cluster tilted algebras of type $\tAn$ and their relations
(as shown in \cite{ABCP}). In Section 5 we first briefly review the fundamental results on
derived equivalences. Afterwards, we prove our main result, i.e., we show, when two cluster tilted
algebras of type $\tAn$ are derived equivalent.

\bigskip

\noindent
{\bf Acknowledgement.} The author would like to thank Thorsten Holm for many helpful suggestions
and discussions about the topics of this work.

\begin{figure}
\centering
\includegraphics[scale=0.65]{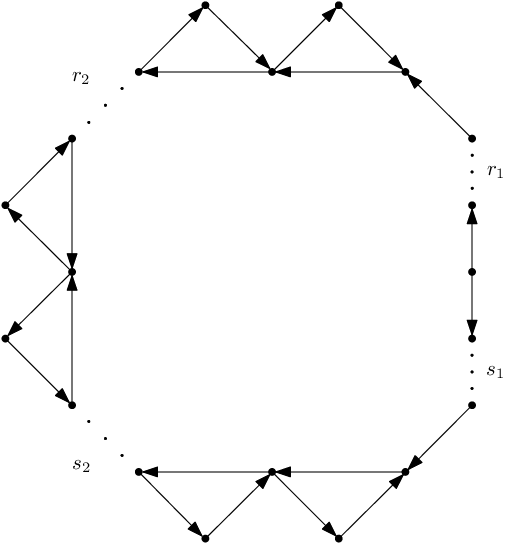}
\caption{Normal form for quivers in $\mQ_n$.}
\label{normal_form}
\end{figure}


\section{Quiver mutations}

First, we present the basic notions for quiver mutations.
A {\em quiver} is a finite directed graph $Q$, consisting of a finite set of vertices $Q_0$
and a finite set of arrows $Q_1$ between them.

Let $Q$ be a quiver and $K$ be an algebraically closed field. We can form the {\em path algebra}
$KQ$, where the basis of $KQ$ is given by all paths in $Q$, including trivial paths $e_i$ of length
zero at each vertex $i$ of $Q$.
Multiplication in $KQ$ is defined by concatenation of paths. Our convention is to read paths from
right to left. For any path $\alpha$ in $Q$ let $s(\alpha)$ denote its start vertex and $t(\alpha)$
its end vertex. Then the product of two paths $\alpha$ and $\beta$ is defined to be the
concatenated path $\alpha \beta$ if $s(\alpha) = t(\beta)$. The unit element of $KQ$ is the sum
of all trivial paths, i.e., $1_{KQ} = \sum\limits_{i \in Q_0} e_i$.

\smallskip

We now recall the definition of {\em quiver mutation} which was introduced by Fomin and
Zelevinsky in \cite{Fomin-Zelevinsky}.

\begin{dfn}
Let $Q$ be a quiver without loops and oriented $2$-cycles.
The {\em mutation} of $Q$ at a vertex $k$ to a new quiver $Q^*$ can be described as follows:
\begin{enumerate}
\item Add a new vertex $k^*$.
\item If there are $r>0$ arrows $i \rightarrow k$, $s>0$ arrows $k \rightarrow j$ and $t \in \Z$
arrows $j \rightarrow i$ in $Q$,
there are $t-rs$ arrows $j \rightarrow i$ in $Q^*$.
(Here, a negative number of arrows means arrows in the opposite direction.)
\item For any vertex $i$ replace all arrows from $i$ to $k$ with arrows from $k^*$ to $i$,
and replace all arrows from $k$ to $i$ with arrows from $i$ to $k^*$.
\item Remove the vertex $k$.
\end{enumerate}
\end{dfn}

Note that mutation at sinks or sources only means changing the direction of all incoming
or outgoing arrows.
Two quivers are called {\em mutation equivalent} ({\em sink/source equivalent}) if one
can be obtained from the other by a finite sequence of mutations (at sinks and/or sources).
The {\em mutation class} of a quiver $Q$ is the class of all quivers mutation equivalent to $Q$.


\section{Mutation classes of $\tAn-$quivers} \label{section_mutation_class}

\begin{rmk}
Quivers of type $\tAn$ are just cycles with $n+1$ vertices.
If the cycle is oriented, then we get the mutation class of $D_{n+1}$ (see \cite{Derksen-Owen}, \cite{FST},
\cite{Fomin-ZelevinskyII} or Type IV in type $D$ in \cite{Vatne}).
If the cycle is non-oriented, we get what we call the mutation classes of $\tAn$.
\end{rmk}

First, we have to fix one drawing of this non-oriented cycle, i.e., one embedding into the plane.
Thus, we can speak of clockwise and anti-clockwise oriented arrows.
But we have to consider that this notation is only unique up to reflection of the cycle, i.e.,
up to changing the roles of clockwise and anti-clockwise oriented arrows.

\begin{lem} 
[Fomin, Shapiro and Thurston, Lemma 6.8 in \cite{FST}] \label{mutation_equivalence}
Let $C_1$ and $C_2$ be two non-oriented cycles, so that in $C_1$ (resp. $C_2$)
there are $s$ (resp. $\tilde{s}$) arrows oriented in the clockwise direction and $r$
(resp. $\tilde{r}$) arrows oriented in the anti-clockwise direction.
Then $C_1$ and $C_2$ are mutation equivalent if and only if the unordered pairs
$\{r,s\}$ and $\{\tilde{r}, \tilde{s}\}$ coincide.
\end{lem}

Thus, two non-oriented cycles of length $n+1$ are mutation equivalent if and only if
they have the same parameters $r$ and $s$ (up to changing the roles of $r$ and $s$).

\bigskip

Next we will provide an explicit description of the mutation classes of $\tAn-$quivers.
For this we need a description of the mutation class of quivers of type $A_k$ and we use
the following one which is given in \cite{Buan-Vatne}:

\begin{itemize}
\item each quiver has $k$ vertices,
\item all non-trivial cycles are oriented and of length 3,
\item a vertex has at most four incident arrows,
\item if a vertex has four incident arrows, then two of them belong to one
oriented $3$-cycle, and the other two belong to another oriented $3$-cycle,
\item if a vertex has three incident arrows, then two of them belong to an oriented
$3$-cycle, and the third arrow does not belong to any oriented $3$-cycle.
\end{itemize}
By a {\em cycle} in the second condition we mean a cycle in the underlying
graph, not passing through the same edge twice. In particular,
this condition excludes multiple arrows.
Note that another description of the mutation class of quivers of type $A$
is given in \cite{Seven}.

\medskip

Now we can formulate the description of the mutation classes of $\tAn-$quivers
which is a similar description as for Type IV in type $D$ in \cite{Vatne}.

\begin{dfn} \label{description_mutation_class}
Let $\mQ_n$ be the class of connected quivers with $n+1$ vertices which satisfy the
following conditions (see Figure~\ref{mutation_class_quiver} for an illustration):
\begin{itemize}
\item[\rn1)] There exists precisely one full subquiver which is a non-oriented cycle of
length $\geq 2$. Thus, if the length is two, it is a double arrow.
\item[\rn2)] For each arrow $x \xrightarrow{\alpha} y$ in this non-oriented cycle, there may
(or may not) be a vertex $z_{\alpha}$ which is not on the non-oriented cycle, such that
there is an oriented $3$-cycle of the form
\begin{center}
\includegraphics[scale=0.9]{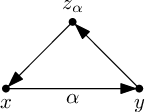}
\end{center}
Apart from the arrows of these oriented $3$-cycles there are no other arrows
incident to vertices on the non-oriented cycle.
\item[\rn3)] If we remove all vertices in the non-oriented cycle and their incident
arrows, the result is a disjoint union of quivers $Q_1, Q_2, \dots$, one for each
$z_{\alpha}$ (which we call $Q_{\alpha}$ in the following).
These are quivers of type $A_{k_\alpha}$ for ${k_\alpha} \geq 1$, and the
vertices $z_{\alpha}$ have at most two incident arrows in these quivers.
Furthermore, if a vertex $z_{\alpha}$ has two incident arrows in such a quiver,
then $z_{\alpha}$ is a vertex in an oriented $3$-cycle in $Q_{\alpha}$.
\end{itemize}
\end{dfn}

\begin{figure}
\centering
\includegraphics[scale=0.8]{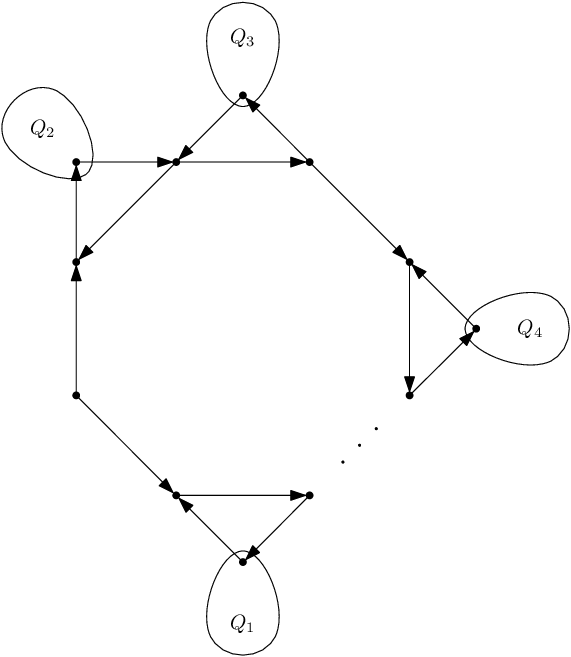}
\caption{Quiver in $\mQ_n$.}
\label{mutation_class_quiver}
\end{figure}

Our convention is to choose only one of the double arrows to be part of the
oriented $3$-cycle in the following case:
\begin{center}
\includegraphics[scale=0.9]{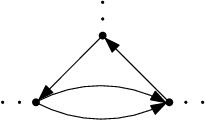}
\end{center}

\begin{nt}
Note that whenever we draw an edge \includegraphics[scale=0.9]{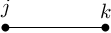}
the direction of the arrow between $j$ and $k$ is not important for this situation;
and whenever we draw a cycle
\begin{center}
\includegraphics[scale=0.9]{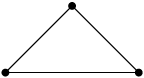}
\end{center}
it is an oriented $3$-cycle.
\end{nt}

\begin{lem}
$\mQ_n$ is closed under quiver mutation.
\end{lem}

\begin{proof}
Let $Q$ be a quiver in $\mQ_n$ and let $i$ be some vertex of $Q$.
The subquivers $Q_1$ and $Q_2$ highlighted in the pictures
are quivers of type $A$.

\medskip

If $i$ is a vertex in one of the quivers $Q_{\alpha}$ of type $A$,
but not one of the vertices $z_{\alpha}$ connecting this quiver of type $A$
to the rest of the quiver $Q$, then the mutation at $i$ leads to a quiver
$Q^* \in \mQ_n$ since type $A$ is closed under quiver mutation.

It therefore suffices to check what happens when we mutate at the other vertices,
and we will consider the following four cases.

\medskip

1) Let $i$ be one of the vertices $z_{\alpha}$, hence not on the non-oriented
cycle. For the situation where the quiver $Q_{\alpha}$ of type $A$ attached to
$z_{\alpha}$ consists only of one vertex, we can look at the first mutated quiver
in case 2) below since quiver mutation is an involution.
Thus, we have the following three cases,

\medskip

\begin{minipage}{0.96\linewidth}
\hfill
\begin{minipage}{0.3\linewidth}
\begin{center}
\includegraphics[scale=0.8]{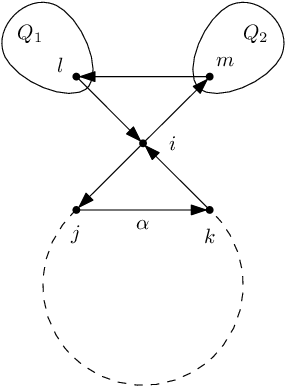}
\end{center}
\end{minipage}
\hfill
\begin{minipage}{0.03\linewidth}
or
\end{minipage}
\hfill
\begin{minipage}{0.3\linewidth}
\begin{center}
\includegraphics[scale=0.8]{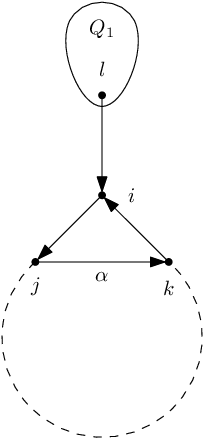}
\end{center}
\end{minipage}
\hfill
\begin{minipage}{0.03\linewidth}
or
\end{minipage}
\hfill
\begin{minipage}{0.3\linewidth}
\begin{center}
\includegraphics[scale=0.8]{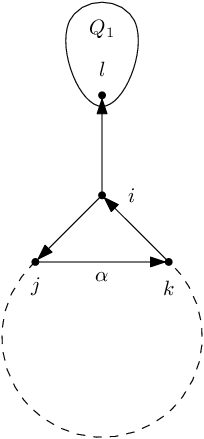}
\end{center}
\end{minipage}
\end{minipage}

\bigskip

Then the mutation at $i$ leads to the following three quivers which have
a non-oriented cycle one arrow longer than for $Q$,
and this is indeed a non-oriented cycle since the arrows
$j \rightarrow i \rightarrow k$ have the same orientation as $\alpha$ had
before.

\medskip

\begin{minipage}{0.96\linewidth}
\hfill
\begin{minipage}{0.34\linewidth}
\begin{center}
\includegraphics[scale=0.8]{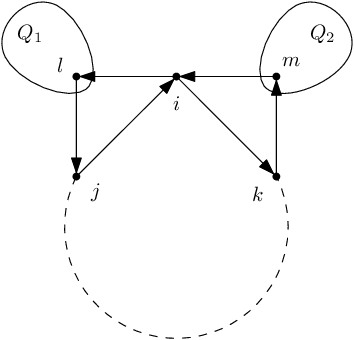}
\end{center}
\end{minipage}
\hfill
\begin{minipage}{0.03\linewidth}
or
\end{minipage}
\hfill
\begin{minipage}{0.28\linewidth}
\begin{center}
\includegraphics[scale=0.8]{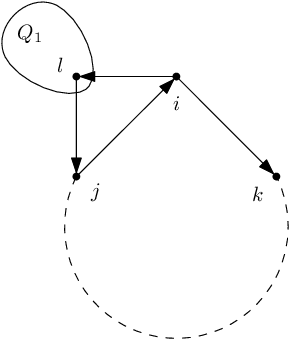}
\end{center}
\end{minipage}
\hfill
\begin{minipage}{0.03\linewidth}
or
\end{minipage}
\hfill
\begin{minipage}{0.28\linewidth}
\begin{center}
\includegraphics[scale=0.8]{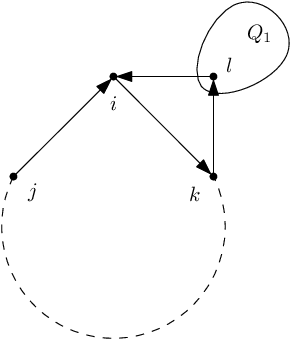}
\end{center}
\end{minipage}
\end{minipage}

\bigskip

The vertices $l$ and $m$ have at most two incident arrows in the quivers
$Q_1$ and $Q_2$ since they had at most four resp. three incident arrows in $Q$
(see the description of quivers of type $A$).
Furthermore, if $l$ or $m$ has two incident arrows in the quiver $Q_1$ or $Q_2$,
then these two arrows form an oriented $3$-cycle as in $Q$.
Thus, the mutated quiver $Q^*$ is also in $\mQ_n$.

\medskip

2) Let $i$ be a vertex on the non-oriented cycle, and not part of any
oriented $3$-cycle. Then the following three cases can occur,

\medskip

\begin{minipage}{0.96\linewidth}
\hfill
\begin{minipage}{0.3\linewidth}
\begin{center}
\includegraphics[scale=0.8]{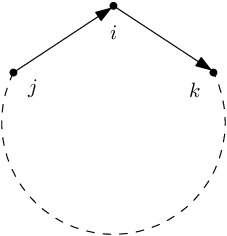}
\end{center}
\end{minipage}
\hfill
\begin{minipage}{0.03\linewidth}
or
\end{minipage}
\hfill
\begin{minipage}{0.3\linewidth}
\begin{center}
\includegraphics[scale=0.8]{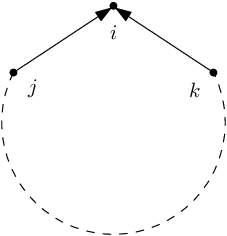}
\end{center}
\end{minipage}
\hfill
\begin{minipage}{0.03\linewidth}
or
\end{minipage}
\hfill
\begin{minipage}{0.3\linewidth}
\begin{center}
\includegraphics[scale=0.8]{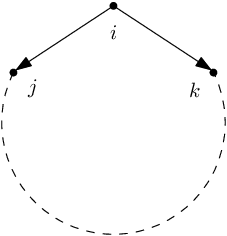}
\end{center}
\end{minipage}
\end{minipage}

\bigskip

and the mutation at $i$ leads to

\medskip

\begin{minipage}{0.96\linewidth}
\hfill
\begin{minipage}{0.3\linewidth}
\begin{center}
\includegraphics[scale=0.8]{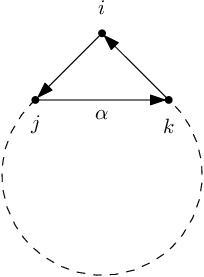}
\end{center}
\end{minipage}
\hfill
\begin{minipage}{0.03\linewidth}
or
\end{minipage}
\hfill
\begin{minipage}{0.3\linewidth}
\begin{center}
\includegraphics[scale=0.8]{non_oriented_3}
\end{center}
\end{minipage}
\hfill
\begin{minipage}{0.03\linewidth}
or
\end{minipage}
\hfill
\begin{minipage}{0.3\linewidth}
\begin{center}
\includegraphics[scale=0.8]{non_oriented_2}
\end{center}
\end{minipage}
\end{minipage}

\bigskip

If $i$ is a sink or a source in $Q$, the non-oriented cycle in $Q^*$ is of the same
length as before and $Q^*$ is in $\mQ_n$.
If there is a path $j \rightarrow i \rightarrow k$ in $Q$, then the mutation at $i$
leads to a quiver $Q^*$ which has a non-oriented cycle one arrow shorter than in $Q$.

Note that in this case the non-oriented cycle in $Q$ consists of at least three arrows
and thus, the non-oriented cycle in $Q^*$ has at least two arrows.
Thus, the mutated quiver $Q^*$ is also in $\mQ_n$.

\medskip

3) Let $i$ be a vertex on the non-oriented cycle which is part of exactly one
oriented $3$-cycle. Then four cases can occur, but two of them have been dealt with
by the second and third mutated quiver in case 1) since quiver mutation is an involution.
Thus, we only have to consider the following two situations and their special
cases where the non-oriented cycle is a double arrow.

\medskip

\begin{minipage}{0.96\linewidth}
\hfill
\begin{minipage}{0.44\linewidth}
\begin{center}
\includegraphics[scale=0.8]{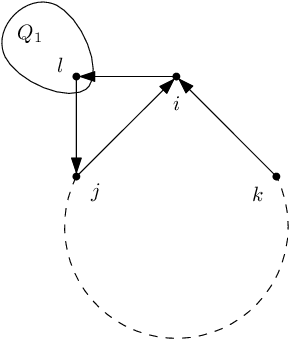}
\end{center}
\end{minipage}
\hfill
\begin{minipage}{0.1\linewidth}
$\substack{\textrm{mutation } \\ \longleftrightarrow  \\ \textrm{at } i}$
\end{minipage}
\hfill
\begin{minipage}{0.44\linewidth}
\begin{center}
\includegraphics[scale=0.8]{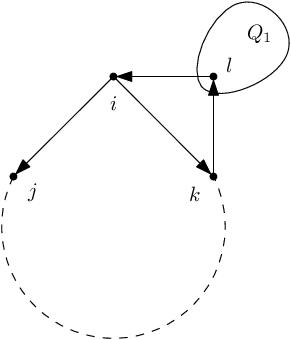}
\end{center}
\end{minipage}
\end{minipage}

\bigskip

\begin{minipage}{0.96\linewidth}
\hfill
\begin{minipage}{0.44\linewidth}
\begin{center}
\includegraphics[scale=0.8]{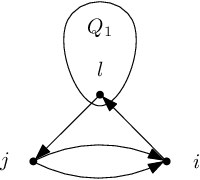}
\end{center}
\end{minipage}
\hfill
\begin{minipage}{0.1\linewidth}
$\substack{\textrm{mutation } \\ \longleftrightarrow  \\ \textrm{at } i}$
\end{minipage}
\hfill
\begin{minipage}{0.44\linewidth}
\begin{center}
\includegraphics[scale=0.8]{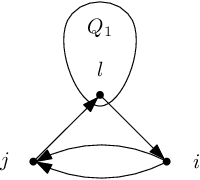}
\end{center}
\end{minipage}
\end{minipage}

\bigskip

After mutating at vertex $i$, the non-oriented cycle has the same length as before.
Moreover, $l$ has the same number of incident arrows as before.
Thus, $Q^*$ is in $\mQ_n$.

\medskip

4) Let $i$ be a vertex on the non-oriented cycle which is part of two
oriented $3$-cycles. Then three cases can occur, but one of them has been dealt with
by the first mutated quiver in case 1).
Thus, we only have to consider the following two situations and their
special cases where the non-oriented cycle is a double arrow.

\medskip

\begin{minipage}{0.96\linewidth}
\hfill
\begin{minipage}{0.44\linewidth}
\begin{center}
\includegraphics[scale=0.8]{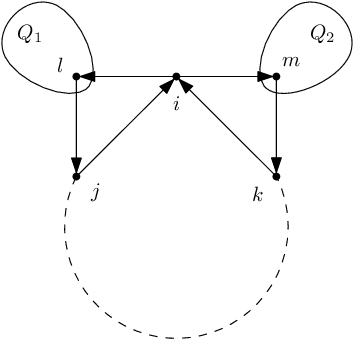}
\end{center}
\end{minipage}
\hfill
\begin{minipage}{0.1\linewidth}
$\substack{\textrm{mutation } \\ \longleftrightarrow  \\ \textrm{at } i}$
\end{minipage}
\hfill
\begin{minipage}{0.44\linewidth}
\begin{center}
\includegraphics[scale=0.8]{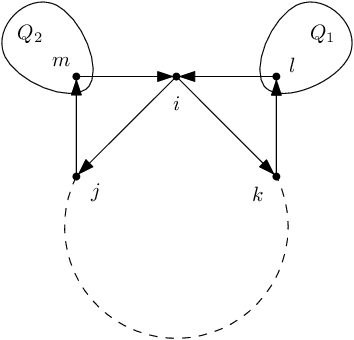}
\end{center}
\end{minipage}
\end{minipage}

\bigskip

\begin{minipage}{0.96\linewidth}
\hfill
\begin{minipage}{0.44\linewidth}
\begin{center}
\includegraphics[scale=0.8]{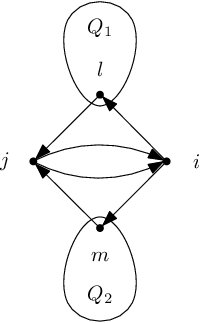}
\end{center}
\end{minipage}
\hfill
\begin{minipage}{0.1\linewidth}
$\substack{\textrm{mutation } \\ \longleftrightarrow  \\ \textrm{at } i}$
\end{minipage}
\hfill
\begin{minipage}{0.44\linewidth}
\begin{center}
\includegraphics[scale=0.8]{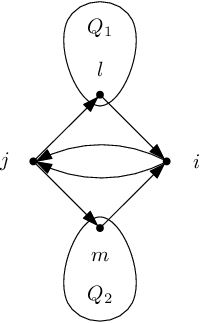}
\end{center}
\end{minipage}
\end{minipage}

\bigskip

After mutating at vertex $i$, the non-oriented cycle has the same length as before.
Moreover, $l$ and $m$ have the same number of incident arrows as before.
Thus, the mutated quiver $Q^*$ is in~$\mQ_n$.
\end{proof}

\begin{rmk}
It is easy to see that all orientations of a circular quiver of type $\tAn$ are in $\mQ_n$
(except the oriented case; but this leads to the mutation class of $D_{n+1}$).
Since $\mQ_n$ is closed under quiver mutation every quiver mutation equivalent to some
quiver of type $\tAn$ is in $\mQ_n$, too.
\end{rmk}

Now we fix one drawing of a quiver $Q \in \mQ_n$, i.e., one embedding into the plane,
without arrow-crossing. Thus, we can again speak of clockwise and anti-clockwise
oriented arrows of the non-oriented cycle. But we have to consider that this
notation is only unique up to reflection of the non-oriented cycle, i.e., up to
changing the roles of clockwise and anti-clockwise oriented arrows.
We define four parameters $r_1, \ r_2, \ s_1$ and $s_2$ for a quiver $Q \in \mQ_n$ as follows:

\begin{dfn} \label{parameters}
Let $r_1$ be the number of arrows which are not part of any oriented $3$-cycle and
which fulfill one of the following two conditions:
\begin{enumerate}
\item These arrows are part of the non-oriented cycle and they are oriented in the anti-clockwise
direction, see the left picture of Figure~\ref{parameter1}.
\item These arrows are not part of the non-oriented cycle, but they are attached to an oriented
$3$-cycle $C$ which shares one arrow $\alpha$ with the non-oriented cycle and $\alpha$ is
oriented in the anti-clockwise direction, see the right picture of Figure~\ref{parameter1}.

In this sense, 'attached' means that these arrows are part of the quiver $Q_{\alpha}$
of type $A$ which shares the vertex $z_{\alpha}$ with the cycle $C$ (see Definition
\ref{description_mutation_class}).
\end{enumerate}

Let $r_2$ be the number of oriented $3$-cycles which fulfill one of the
following two conditions:

\begin{enumerate}
\item These cycles share one arrow $\alpha$ with the non-oriented cycle and $\alpha$
is oriented in the anti-clockwise direction, see the left picture of Figure~\ref{parameter2}.
\item These cycles are attached to an oriented $3$-cycle $C$ sharing one arrow $\alpha$
with the non-oriented cycle and $\alpha$ is oriented in the anti-clockwise direction,
see the right picture of Figure~\ref{parameter2}.
Here, 'attached' is in the same sense as above.
\end{enumerate}

Similarly we define the parameters $s_1$ and $s_2$ with 'clockwise' instead of
'anti-clockwise'.
\end{dfn}

\begin{figure}
\centering
\setlength{\tabcolsep}{2em}
\begin{tabular}{cc}
\includegraphics[scale=0.8]{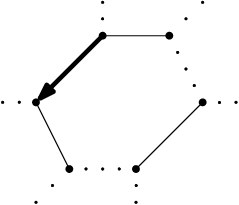}
&
\includegraphics[scale=0.8]{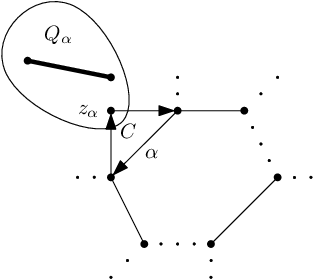}
\end{tabular}
\caption{Illustration for parameter $r_1$.}
\label{parameter1}
\end{figure}

\begin{figure}
\centering
\setlength{\tabcolsep}{2em}
\begin{tabular}{cc}
\includegraphics[scale=0.8]{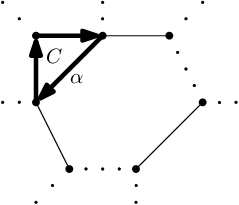}
&
\includegraphics[scale=0.8]{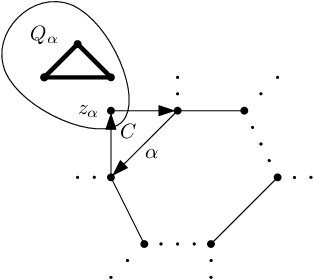}
\end{tabular}
\caption{Illustration for parameter $r_2$.}
\label{parameter2}
\end{figure}

\begin{eg}
We denote the arrows which count for the parameter $r_1$ by \includegraphics[scale=0.9]{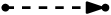}
and the arrows which count for $s_1$ by \includegraphics[scale=0.9]{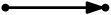}.
Furthermore, the oriented $3$-cycles of $r_2$ are denoted by \includegraphics[scale=0.9]{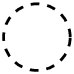}
and the oriented $3$-cycles of $s_2$ are denoted by \includegraphics[scale=0.9]{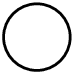}.
Let $Q \in \mQ_{16}$ be a quiver of the following form
\begin{center}
\includegraphics[scale=0.9]{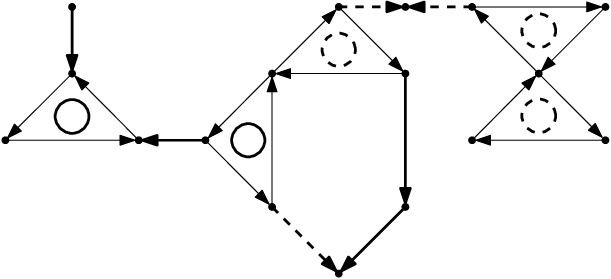}
\end{center}
Then we get $r_1=3, \ r_2=3, \ s_1=4$ and $s_2=2$.
\end{eg}

\begin{lem} \label{mutate_to_normalform}
If $Q_1$ and $Q_2$ are quivers in $\mQ_n$, and $Q_1$ and $Q_2$ have the same parameters
$r_1, \ r_2, \ s_1$ and $s_2$ (up to changing the roles of $r_i$ and $s_i$, $i\in \{1,2\}$),
then $Q_2$ can be obtained from $Q_1$ by iterated mutation, where all the intermediate
quivers have the same parameters as well.
\end{lem}

\begin{proof}
It is enough to show that all quivers in $\mQ_n$ with parameters $r_1, \ r_2, \ s_1$
and $s_2$ can be mutated to a quiver in {\em normal form}, see Figure~\ref{normal_form},
without changing the parameters $r_1, \ r_2, \ s_1$ and $s_2$.
In such a quiver, $r_1$ is the number of anti-clockwise arrows in the non-oriented cycle
which do not share any arrow with an oriented $3$-cycle and
$s_1$ is the number of clockwise arrows in the non-oriented cycle which do not share any
arrow with an oriented $3$-cycle.
Furthermore, $r_2$ is the number of oriented $3$-cycles sharing one arrow $\alpha$ with
the non-oriented cycle and $\alpha$ is oriented in the anti-clockwise direction and
$s_2$ is the number of oriented $3$-cycles sharing one arrow $\beta$ with the non-oriented
cycle and $\beta$ is oriented in the clockwise direction (see Definition~\ref{parameters}).

\medskip

We divide this process into five steps.

\medskip

\begin{minipage}{0.96\linewidth}
\begin{minipage}{0.13\linewidth}
Step $1$:
\end{minipage}%
\begin{minipage}[t]{0.87\linewidth}
Let $Q$ be a quiver in $\mQ_n$.
We move all oriented $3$-cycles of $Q$ sharing no arrow with the non-oriented cycle
towards the oriented $3$-cycle which is attached to them and which shares one arrow
with the non-oriented cycle.
\end{minipage}
\end{minipage}

\bigskip

\begin{minipage}{0.96\linewidth}
\begin{minipage}{0.13\linewidth}
Method:
\end{minipage}%
\begin{minipage}[t]{0.87\linewidth}
Let $C$ and $C^{\prime}$ be a pair of neighbouring oriented $3$-cycles in $Q$
(i.e., no arrow in the path between them is part of an oriented $3$-cycle) such
that the length of the path between them is at least one.
We want to move $C$ and $C^{\prime}$ closer together by mutation.
\end{minipage}
\end{minipage}

\bigskip

\begin{minipage}{0.96\linewidth}
\begin{minipage}{0.13\linewidth}
\
\end{minipage}%
\begin{minipage}[c]{0.87\linewidth}
\begin{center}
\includegraphics[scale=0.9]{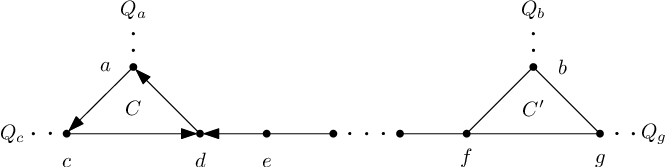}
\end{center}
\end{minipage}
\end{minipage}

\bigskip

\begin{minipage}{0.96\linewidth}
\begin{minipage}{0.13\linewidth}
\
\end{minipage}%
\begin{minipage}[c]{0.87\linewidth}
In the picture the $Q_i$ are subquivers of $Q$.
Mutating at $d$ will produce a quiver $Q^*$ which looks like this:
\end{minipage}
\end{minipage}

\bigskip

\begin{minipage}{0.96\linewidth}
\begin{minipage}{0.13\linewidth}
\ 
\end{minipage}%
\begin{minipage}[t]{0.87\linewidth}
\begin{center}
\includegraphics[scale=0.9]{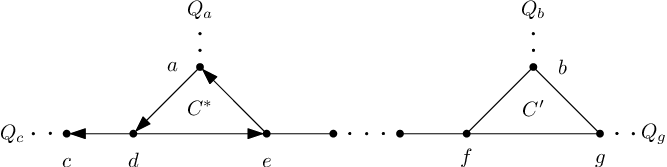}
\end{center}
\end{minipage}
\end{minipage}

\bigskip

\begin{minipage}{0.96\linewidth}
\begin{minipage}{0.13\linewidth}
\ 
\end{minipage}%
\begin{minipage}[t]{0.87\linewidth}
Thus, the length of the path between $C^*$ and $C^{\prime}$ decreases by $1$ and
there is a path of length one between $C^*$ and $Q_c$.
Note that the arguments for a quiver with arrow $d \rightarrow e$ are analogous
and that these mutations can also be used if the arrows between $d$ and $f$ are
part of the non-oriented cycle (see Step $4$).
\end{minipage}
\end{minipage}

\bigskip

\begin{minipage}{0.96\linewidth}
\begin{minipage}{0.13\linewidth}
\
\end{minipage}%
\begin{minipage}[t]{0.87\linewidth}
In this procedure, the parameters $r_1, \ r_2, \ s_1$ and $s_2$ are left unchanged
since we are not changing the number of arrows and the number of oriented $3$-cycles
which are attached to an oriented $3$-cycle sharing one arrow with the non-oriented cycle.
\end{minipage}
\end{minipage}

\bigskip

\begin{minipage}{0.96\linewidth}
\begin{minipage}{0.13\linewidth}
Step $2$:
\end{minipage}%
\begin{minipage}[t]{0.87\linewidth}
We move all oriented $3$-cycles onto the non-oriented cycle.
\end{minipage}
\end{minipage}

\bigskip

\begin{minipage}{0.96\linewidth}
\begin{minipage}{0.13\linewidth}
Method:
\end{minipage}%
\begin{minipage}[t]{0.87\linewidth}
Let $C$ be an oriented $3$-cycle which shares one vertex $z_{\alpha}$ with an oriented
$3$-cycle $C_{\alpha}$ sharing an arrow $\alpha$ with the non-oriented cycle.
Then we mutate at the vertex $z_{\alpha}$:
\end{minipage}
\end{minipage}

\bigskip

\begin{minipage}{0.96\linewidth}
\begin{minipage}{0.13\linewidth}
\ 
\end{minipage}%
\begin{minipage}[t]{0.87\linewidth}
\begin{minipage}{0.32\linewidth}
\begin{center}
\includegraphics[scale=0.8]{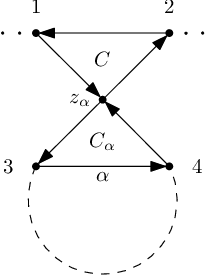}
\end{center}
\end{minipage}
\hfill
\begin{minipage}{0.1\linewidth}
$\substack{\textrm{mutation } \\ \leadsto \\ \textrm{at } z_{\alpha}}$
\end{minipage}
\hfill
\begin{minipage}{0.45\linewidth}
\begin{center}
\includegraphics[scale=0.8]{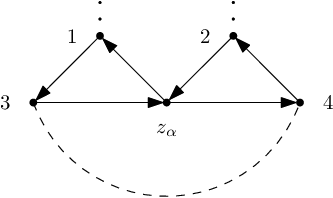}
\end{center}
\end{minipage}
\end{minipage}
\end{minipage}

\bigskip

\begin{minipage}{0.96\linewidth}
\begin{minipage}{0.13\linewidth}
\ 
\end{minipage}%
\begin{minipage}[t]{0.87\linewidth}
Hence, both of the oriented $3$-cycles share one arrow with the non-oriented
cycle and these arrows are oriented as $\alpha$ was before.
Thus, the parameters $r_1, \ r_2, \ s_1$ and $s_2$ are left unchanged.
Furthermore, the length of the non-oriented cycle increases by $1$.
By iterated mutation of that kind, we produce a quiver $Q^*$, where all the oriented
$3$-cycles share an arrow with the non-oriented cycle.
\end{minipage}
\end{minipage}

\bigskip

\begin{minipage}{0.96\linewidth}
\begin{minipage}{0.13\linewidth}
Step $3$:
\end{minipage}%
\begin{minipage}[t]{0.87\linewidth}
We move all arrows onto the non-oriented cycle.
\end{minipage}
\end{minipage}

\bigskip

\begin{minipage}{0.96\linewidth}
\begin{minipage}{0.13\linewidth}
Method:
\end{minipage}%
\begin{minipage}[t]{0.87\linewidth}
This is a similar process as in Step $2$:
Let $C_{\alpha}$ be an oriented $3$-cycle which shares an arrow $\alpha$ with
the non-oriented cycle.
All arrows attached to $C_{\alpha}$ can be moved into the non-oriented cycle
by iteratively mutating at vertex $z_{\alpha}$. After mutating, all these
arrows have the same orientation as $\alpha$ in the non-oriented cycle.
Thus, the parameters $r_1, \ r_2, \ s_1$ and $s_2$ are left unchanged.
\end{minipage}
\end{minipage}

\bigskip

\begin{minipage}{0.96\linewidth}
\begin{minipage}{0.13\linewidth}
Step $4$:
\end{minipage}%
\begin{minipage}[t]{0.87\linewidth}
Move oriented $3$-cycles along the non-oriented cycle.
\end{minipage}
\end{minipage}

\bigskip

\begin{minipage}{0.96\linewidth}
\begin{minipage}{0.13\linewidth}
Method:
\end{minipage}%
\begin{minipage}[t]{0.87\linewidth}
First, we number all oriented $3$-cycles by $C_1, \dots, C_{r_2+s_2}$ in such
a way that $C_{i+1}$ follows $C_i$ in the anti-clockwise direction.
As in Step $1$, we can move an oriented $3$-cycle $C_i$ towards $C_{i+1}$,
without changing the orientation of the arrows, i.e., without changing the
parameters $r_1, \ r_2, \ s_1$ and $s_2$.
\end{minipage}
\end{minipage}

\bigskip

\begin{minipage}{0.96\linewidth}
\begin{minipage}{0.13\linewidth}
\
\end{minipage}%
\begin{minipage}[t]{0.87\linewidth}
Note that if the non-oriented cycle includes the vertex $a$ in the pictures
of Step $1$, the arrows between the two cycles move to the top of $C_{i+1}$,
i.e., they are no longer part of the non-oriented cycle.
However, we can reverse their directions by mutating at the new sinks or sources
and insert these arrows into the non-oriented cycle between $C_{i+1}$ and $C_{i+2}$
by mutations like in Step $3$ (if $C_{i+2}$ exists).
\end{minipage}
\end{minipage}

\bigskip

\begin{minipage}{0.96\linewidth}
\begin{minipage}{0.13\linewidth}
\ 
\end{minipage}%
\begin{minipage}[t]{0.87\linewidth}
Doing this iteratively, we produce a quiver $Q^*$ as in Figure~\ref{form_step4},
with $r_1 + s_1$ arrows which are not part of any oriented $3$-cycle and
$r_2 + s_2$ oriented $3$-cycles sharing one arrow with the non-oriented cycle.
\end{minipage}
\end{minipage}

\begin{figure}
\centering
\includegraphics[scale=0.7]{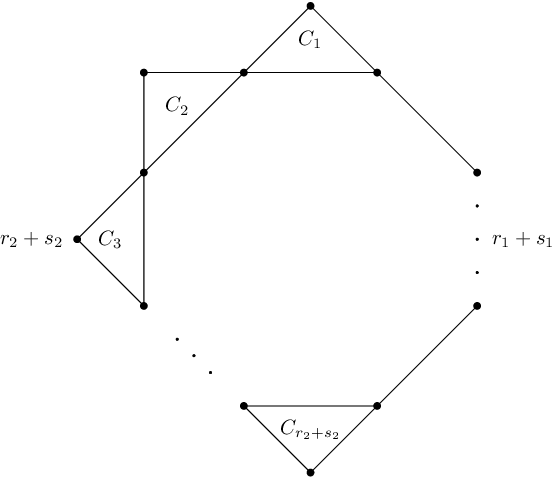}
\caption{Normal form of Step $4$.}
\label{form_step4}
\end{figure}

\bigskip

\begin{minipage}{0.96\linewidth}
\begin{minipage}{0.13\linewidth}
Step $5$:
\end{minipage}%
\begin{minipage}[t]{0.87\linewidth}
Changing orientation on the non-oriented cycle to the orientation of Figure~\ref{normal_form}.
\end{minipage}
\end{minipage}

\bigskip

\begin{minipage}{0.96\linewidth}
\begin{minipage}{0.13\linewidth}
Method:
\end{minipage}%
\begin{minipage}[t]{0.87\linewidth}
The part of the non-oriented cycle without oriented $3$-cycles can be
moved to the desired orientation of Figure~\ref{normal_form} via sink/source mutations,
without mutating at the `end' vertices which are attached to oriented $3$-cycles.
Thus, the parameters $r_1$ and $s_1$ are left unchanged.
\end{minipage}
\end{minipage}

\bigskip

\begin{minipage}{0.96\linewidth}
\begin{minipage}{0.13\linewidth}
\ 
\end{minipage}%
\begin{minipage}[t]{0.87\linewidth}
Each oriented $3$-cycle shares one arrow with the non-oriented cycle.
If all of these arrows are oriented in the same direction, the quiver is in
the required form. Thus, we can assume that there are at least two arrows of
two oriented $3$-cycles $C_i$ and $C_{i+1}$ which are oriented in opposite directions.
If we mutate at the connecting vertex of $C_i$ and $C_{i+1}$, the directions
of these arrows are changed:
\end{minipage}
\end{minipage}

\bigskip

\begin{minipage}{0.96\linewidth}
\begin{minipage}{0.13\linewidth}
\ 
\end{minipage}%
\begin{minipage}[c]{0.87\linewidth}
\begin{minipage}{0.43\linewidth}
\begin{center}
\includegraphics[scale=0.8]{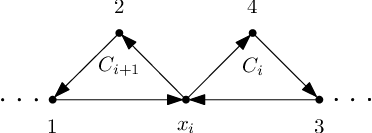}
\end{center}
\end{minipage}
\hfill
\begin{minipage}{0.1\linewidth}
$\substack{\textrm{mutation} \\ \leadsto \\ \textrm{at } x_i}$
\end{minipage}
\hfill
\begin{minipage}{0.43\linewidth}
\begin{center}
\includegraphics[scale=0.8]{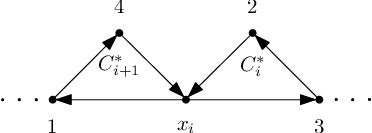}
\end{center}
\end{minipage}
\end{minipage}
\end{minipage}

\bigskip

\begin{minipage}{0.96\linewidth}
\begin{minipage}{0.13\linewidth}
\ 
\end{minipage}%
\begin{minipage}[c]{0.87\linewidth}
\begin{minipage}{0.43\linewidth}
\begin{center}
\includegraphics[scale=0.8]{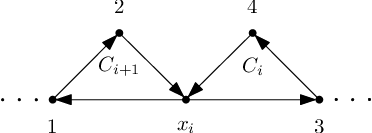}
\end{center}
\end{minipage}
\hfill
\begin{minipage}{0.1\linewidth}
$\substack{\textrm{mutation} \\ \leadsto \\ \textrm{at } x_i}$
\end{minipage}
\hfill
\begin{minipage}{0.43\linewidth}
\begin{center}
\includegraphics[scale=0.8]{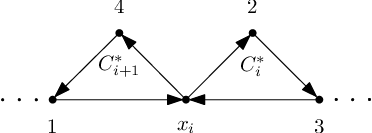}
\end{center}
\end{minipage}
\end{minipage}
\end{minipage}

\bigskip

\begin{minipage}{0.96\linewidth}
\begin{minipage}{0.13\linewidth}
\ 
\end{minipage}%
\begin{minipage}[t]{0.87\linewidth}
Hence, these mutations act like sink/source mutations at the non-oriented
cycle and the parameters $r_2$ and $s_2$ are left unchanged.
Thus, we can mutate at such connecting vertices as in the part without
oriented $3$-cycles to reach the desired orientation of Figure~\ref{normal_form}.
\end{minipage}
\end{minipage}

\end{proof}

\begin{thm} \label{one_part_of_mutation_classes}
Let $Q \in \mQ_n$ with parameters $r_1, \ r_2, \ s_1$ and $s_2$.
Then $Q$ is mutation equivalent to a non-oriented cycle of length $n+1$ with
parameters $r=r_1+2r_2$ and $s=s_1+2s_2$.
\end{thm}

\begin{proof}
We can assume that $Q$ is in normal form (see Lemma \ref{mutate_to_normalform})
and we label the vertices $z_{\alpha}$ as follows:
\begin{center}
\includegraphics[scale=0.75]{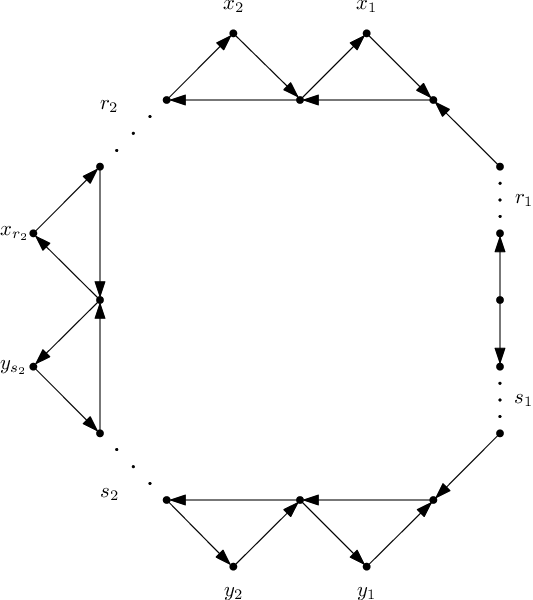}
\end{center}

Mutation at the vertex $x_i$ of an oriented $3$-cycle
\includegraphics[scale=0.9]{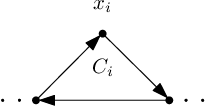}
leads to two arrows of the following form
\includegraphics[scale=0.9]{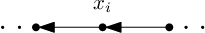} \ .

\smallskip

Thus, after mutating at all the $x_i$, the parameter $r_2$ is zero and
we have a new parameter $r = r_1+2r_2$. Similarly, we get $s = s_1+2s_2$.
Hence, mutating at all the $x_i$ and $y_i$ leads to a quiver with underlying
graph $\tAn$ as follows:
\begin{center}
\includegraphics[scale=0.835]{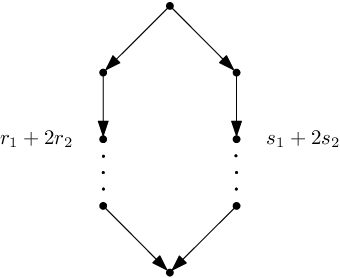}
\end{center}
Since there is a non-oriented cycle in every $Q \in \mQ_n$, these parameters
$r$ and $s$ are non-zero.
Thus, the cycle above is also non-oriented.
Hence, $Q$ is mutation equivalent to some quiver of type $\tAn$ with parameters
$r = r_1+2r_2$ and $s = s_1+2s_2$.
\end{proof}

\begin{cor} \label{corollary}
Let $Q_1, Q_2 \in \mQ_n$ with parameters $r_1, \ r_2, \ s_1$ and $s_2$,
respectively $\tilde{r}_1, \ \tilde{r}_2, \ \tilde{s}_1$ and $\tilde{s}_2$.
If $r_1+2r_2 = \tilde{r}_1 + 2\tilde{r}_2$ and $s_1+2s_2 = \tilde{s}_1 + 2\tilde{s}_2$
(or vice versa), then $Q_1$ is mutation equivalent to $Q_2$.
\end{cor}

\begin{thm} \label{description_mutation_classes}
Let $Q_1, Q_2 \in \mQ_n$ with parameters $r_1, \ r_2, \ s_1$ and $s_2$,
respectively $\tilde{r}_1, \ \tilde{r}_2, \ \tilde{s}_1$ and $\tilde{s}_2$.
Then $Q_1$ is mutation equivalent to $Q_2$ if and only if $r_1+2r_2 = \tilde{r}_1 + 2\tilde{r}_2$
and $s_1+2s_2 = \tilde{s}_1 + 2\tilde{s}_2$ (or $r_1+2r_2 = \tilde{s}_1 + 2\tilde{s}_2$
and $s_1+2s_2 = \tilde{r}_1 + 2\tilde{r}_2$).
\end{thm}
Note that the only-if-part follows from Theorem \ref{one_part_of_mutation_classes}
and Lemma \ref{mutation_equivalence}.


\section{Cluster tilted algebras of type $\tAn$}

In general, cluster tilted algebras arise as endomorphism algebras of cluster-tilting
objects in a cluster category \cite{BMR}.
Since a cluster tilted algebra $A$ of type $\tAn$ is finite dimensional over an
algebraically closed field $K$, 
there exists a quiver $Q$ which is in the mutation classes of $\tAn$ \cite{BMR2}
and an admissible ideal $I$ of the path algebra $KQ$ of $Q$ such that $A \cong KQ/I$.
A non-zero linear combination $k_1\alpha_1 + \dots + k_m\alpha_m, \ k_i \in K\backslash\{0\},$
of paths $\alpha_i$ of length at least two, with the same starting point and the same end point,
is called a {\em relation} in $Q$.
If $m=1$, we call such a relation a {\em zero-relation}.
Any admissible ideal of $KQ$ is generated by a finite set of relations in $Q$.

From \cite{ABCP} and \cite{Assem-Redondo} we know that a cluster tilted algebra $A$ of type
$\tAn$ is gentle. Thus, recall the definition of gentle algebras:

\begin{dfn}
We call $A=KQ/I$ a {\em special biserial algebra} if the following properties hold:
\begin{itemize}
\item[1)] Each vertex of $Q$ is the starting point of at most two arrows and the end point
of at most two arrows.
\item[2)] For each arrow $\alpha$ in $Q$ there is at most one arrow $\beta$ such that
$\alpha\beta \notin I$, and at most one arrow $\gamma$ such that $\gamma\alpha \notin I$.
\end{itemize}
$A$ is {\em gentle} if moreover:
\begin{itemize}
\item[3)] The ideal $I$ is generated by paths of length $2$.
\item[4)] For each arrow $\alpha$ in $Q$ there is at most one arrow $\beta^{\prime}$
with $t(\alpha)=s(\beta^{\prime})$ such that
$\beta^{\prime}\alpha \in I$, and there is at most one arrow $\gamma^{\prime}$ with
$t(\gamma^{\prime}) = s(\alpha)$ such that $\alpha\gamma^{\prime} \in I$.
\end{itemize}
\end{dfn}

Furthermore, all relations in a cluster tilted algebra $A$ of type $\tAn$ occur in
the oriented $3$-cycles, i.e., in cycles of the form
\begin{center}
\includegraphics[scale=0.9]{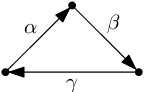}
\end{center}
with (zero-)relations $\alpha\gamma, \ \beta\alpha$ and $\gamma\beta$ (see \cite{ABCP}
and \cite{Assem-Redondo}).

\begin{rmk}
According to our convention in Definition \ref{description_mutation_class} there are
only three (zero-)relations in the following quiver
\begin{center}
\includegraphics[scale=0.9]{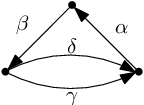}
\end{center}
and here, these are $\alpha\delta, \ \beta\alpha$ and $\delta\beta$.
\end{rmk}

\medskip

For the next section, we need the notion of Cartan matrices of an algebra $A$
(for example see \cite{Holm}).
Let $K$ be a field and $A=KQ/I$. Since $\sum_{i \in Q_0} e_i + I$ is the unit element in
$A$ we get $A=A \cdot 1 = \bigoplus_{i \in Q_0}Ae_i$, hence the (left) $A$-modules
$P_i:= Ae_i$ are the indecomposable projective $A$-modules. The {\em Cartan matrix}
$C=(c_{ij})$ of $A$ is a $|Q_0| \times |Q_0|$-matrix defined by setting
$c_{ij} = \dim_K \mathrm{Hom}_A (P_j,P_i)$. Any homomorphism $\varphi: Ae_j \rightarrow
Ae_i$ of left $A$-modules is uniquely determined by $\varphi(e_j) \in e_jAe_i$, the
$K$-vector space generated by all paths in $Q$ from vertex $i$ to vertex $j$ which
are non-zero in $A$. In particular, we have $c_{ij}=\dim_K e_jAe_i$.

That means, computing entries of the Cartan matrix for $A$ reduces to counting paths
in $Q$ which are non-zero in $A$.


\section{Derived equivalence classification of cluster tilted algebras of type $\tAn$}

First, we briefly review the fundamental results on derived equivalences. 
For a $K$-algebra $A$ the bounded derived category of $A$-modules 
is denoted by $D^b(A)$. Recall that two algebras $A,B$ are called
derived equivalent if $D^b(A)$ and $D^b(B)$ are equivalent
as triangulated categories. By a theorem of Rickard
\cite{Rickard} derived equivalences can be found using the concept of tilting complexes.

\begin{dfn} \label{tilting_complex}
A tilting complex $T$ over $A$ is a bounded complex
of finitely generated projective $A$-modules satisfying the
following conditions:
\begin{itemize}
\item[\rn1)] $\Hom_{D^b(A)}(T,T[i])=0$ for all $i\neq 0$, where $[.]$
denotes the shift functor in $D^b(A)$;
\item[\rn2)] the category $\add(T)$ (i.e. the full subcategory
consisting of direct summands of direct sums of $T$) generates
the homotopy category $K^b(P_A)$ of projective $A$-modules 
as a triangulated category.
\end{itemize}
\end{dfn}

We can now formulate Rickard's celebrated result.

\begin{thm}
[Rickard \cite{Rickard}] \label{Rickard}
Two algebras $A$ and $B$ are derived equivalent if and only if 
there exists a tilting complex $T$ for $A$ such that the
endomorphism algebra $\End_{D^b(A)}(T)\cong B$.
\end{thm}

For calculating the endomorphism algebra $\End_{D^b(A)}(T)$ we can use the
following alternating sum formula which gives a general method for computing 
the Cartan matrix of an endomorphism algebra of a tilting complex from 
the Cartan matrix of the algebra $A$.

\begin{prop}
[Happel \cite{Happel}] \label{Happel}
For an algebra $A$ let $Q = (Q^r)_{r \in \Z}$ and 
$R = (R^s)_{s \in \Z}$ be bounded complexes of projective $A$-modules. Then
$$\sum\limits_i (-1)^i \dim\mathrm{Hom}_{D^b(A)}(Q,R[i])=
\sum\limits_{r,s} (-1)^{r-s}\dim\mathrm{Hom}_A(Q^r,R^s).$$
In particular, if $Q$ and $R$ are direct summands of the same tilting complex
then $$\dim\mathrm{Hom}_{D^b(A)}(Q,R) = \sum\limits_{r,s} 
(-1)^{r-s} \dim\mathrm{Hom}_A(Q^r,R^s).$$
\end{prop}

\begin{lem} \label{derived_equivalences}
Let $A = KQ/I$ be a cluster tilted algebra of type $\tAn$.
Let $r_1, \ r_2, \ s_1$ and $s_2$ be the parameters of $Q$ which are defined in
\ref{parameters}. Then $A$ is derived equivalent to a cluster tilted algebra corresponding
to a quiver in normal form as in Figure~\ref{normal_form}.
\end{lem}

\begin{proof}
First, the number of oriented $3$-cycles with full relations is invariant
under derived equivalence for gentle algebras (see \cite{Holm}), so the number
$r_2+s_2$ is an invariant.
From Proposition $B$ in \cite{Avella-Geiss}, we know that the number of arrows
is also invariant under derived equivalence, so the number $r_1+s_1$ is an invariant, too.
Later, we show in the proof of Theorem \ref{main_result} that the single parameters
$r_1, \ r_2, \ s_1$ and $s_2$ are invariants under derived equivalence.

\smallskip

Our strategy in this proof is to go through the proof of Lemma \ref{mutate_to_normalform}
and define a tilting complex for each mutation in the Steps $1$ and $2$. We can omit the
three other steps since these are just the same situations as in the first two steps.
We show that if we mutate at some vertex of the quiver $Q$ and obtain a quiver $Q^*$,
then the two corresponding cluster tilted algebras are derived equivalent.

\medskip

\begin{minipage}{0.96\linewidth}
\begin{minipage}{0.13\linewidth}
Step $1$
\end{minipage}%
\begin{minipage}[t]{0.87\linewidth}
Let $A$ be a cluster tilted algebra with corresponding quiver
\begin{center}
\includegraphics[scale=0.9]{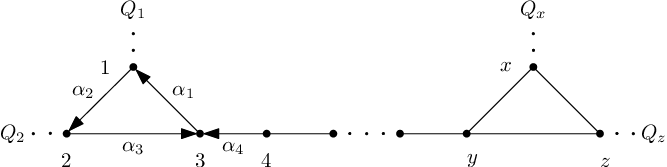}
\end{center}
\end{minipage}
\end{minipage}

\bigskip

\begin{minipage}{0.96\linewidth}
\begin{minipage}{0.13\linewidth}
\ 
\end{minipage}%
\begin{minipage}[t]{0.87\linewidth}
We can compute the Cartan matrix to be
{\scriptsize $\left(\begin{array}[c]{ccccc}1&1&0&0&\dots\\ 0&1&1&0&\dots\\
1&0&1&0&\dots\\ 1&0&1&1&\dots\\  \vdots&\vdots&\vdots&\vdots&\ddots\end{array}\right)$}.
\end{minipage}
\end{minipage}

\bigskip

\begin{minipage}{0.96\linewidth}
\begin{minipage}{0.13\linewidth}
\ 
\end{minipage}%
\begin{minipage}[t]{0.87\linewidth}
Since we deal with left modules and read paths from right to left, a non-zero
path from vertex $i$ to $j$ gives a homomorphism $P_j \rightarrow P_i$ by right
multiplication. Thus, two arrows $\alpha: i \rightarrow j$ and $\beta: j \rightarrow k$
give a path $\beta\alpha$ from $i$ to $k$ and a homomorphism $\alpha\beta :
P_k \rightarrow P_i$.

In the above situation, we have homomorphisms $P_3 \stackrel{\alpha_3}{\longrightarrow }
P_2$ and $P_3 \stackrel{\alpha_4}{\longrightarrow } P_4$.
\end{minipage}
\end{minipage}

\bigskip

\begin{minipage}{0.96\linewidth}
\begin{minipage}{0.13\linewidth}
\ 
\end{minipage}%
\begin{minipage}[c]{0.87\linewidth}
Let $T=\bigoplus_{i=1}^{n+1} T_i$ be the following bounded complex of projective
$A$-modules, where
$T_i : 0 \rightarrow P_i \rightarrow 0, \ i \in \{1,2,4,\dots,n+1\}$, are complexes
concentrated in degree zero and
$T_3 : 0 \rightarrow P_3 \stackrel{(\alpha_3, \alpha_4)}{\longrightarrow} P_2 \oplus P_4
\rightarrow 0$ is a complex concentrated in degrees $-1$ and $0$.

We leave it to the reader to verify that this is indeed a tilting complex.
\end{minipage}
\end{minipage}

\bigskip

\begin{minipage}{0.96\linewidth}
\begin{minipage}{0.13\linewidth}
\ 
\end{minipage}%
\begin{minipage}[c]{0.87\linewidth}
By Rickard's Theorem \ref{Rickard}, $E:=\mathrm{End}_{D^b(A)}(T)$
is derived equivalent to $A$.
Using the alternating sum formula of the Proposition \ref{Happel} of Happel
we can compute the Cartan matrix of $E$ to be
{\scriptsize $\left(\begin{array}[c]{ccccc}1&1&1&0&\dots\\ 0&1&0&0&\dots\\
0&1&1&1&\dots\\ 1&0&0&1&\dots\\  \vdots&\vdots&\vdots&\vdots&\ddots\end{array}\right)$}.
\end{minipage}
\end{minipage}

\bigskip

\begin{minipage}{0.96\linewidth}
\begin{minipage}{0.13\linewidth}
\ 
\end{minipage}%
\begin{minipage}[c]{0.87\linewidth}
We define homomorphisms in $E$ as follows
\begin{center}
\includegraphics[scale=0.9]{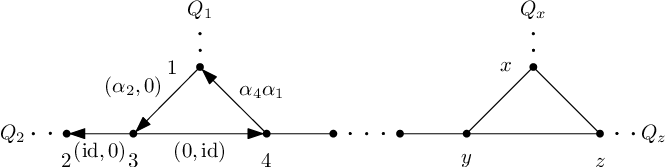}
\end{center}
\end{minipage}
\end{minipage}

\bigskip

\begin{minipage}{0.96\linewidth}
\begin{minipage}{0.13\linewidth}
\ 
\end{minipage}%
\begin{minipage}[c]{0.87\linewidth}
Now we have to check the relations, up to homotopy.

Clearly, the homomorphism $(\alpha_4\alpha_1\alpha_2,0)$ in the oriented $3$-cycle
containing the vertices $1,3$ and $4$ is zero since $\alpha_1\alpha_2$ was zero in $A$.
Furthermore, the composition of $(\alpha_2,0)$ and $(0,\mathrm{id})$ yields a
zero-relation. The last zero-relation in this oriented $3$-cycle is the
concatenation of $(0,\mathrm{id})$ and $\alpha_4\alpha_1$ since this homomorphism
is homotopic to zero:
\end{minipage}
\end{minipage}

\bigskip

\begin{minipage}{0.96\linewidth}
\begin{minipage}{0.13\linewidth}
\ 
\end{minipage}%
\begin{minipage}[c]{0.87\linewidth}
\begin{center}
\includegraphics[scale=0.9]{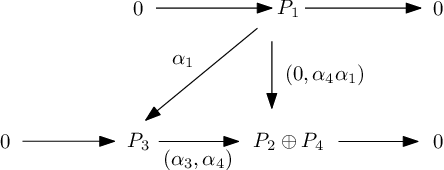}
\end{center}
\end{minipage}
\end{minipage}

\bigskip

\begin{minipage}{0.96\linewidth}
\begin{minipage}{0.13\linewidth}
\ 
\end{minipage}%
\begin{minipage}[c]{0.87\linewidth}
The relations in all the other oriented $3$-cycles of this quiver are
the same as in the quiver of $A$.

Thus, we have defined homomorphisms between the summands of $T$ corresponding
to the arrows of the quiver which we obtain after mutating at vertex $3$ in
the quiver of $A$. We have shown that they satisfy the defining relations
of this algebra and the Cartan matrices agree. Thus, $A$ is derived equivalent
to $E$ and $A^{\mathrm{op}}$ is derived equivalent to $E^{\mathrm{op}}$,
where the quiver of $E$ is the same as the quiver we obtain after mutating
at vertex $3$ in the quiver of $A$. Furthermore, the quivers of $A^{\mathrm{op}}$
and $E^{\mathrm{op}}$ are the quivers in the other case in Step $1$.
\end{minipage}
\end{minipage}

\bigskip

\begin{minipage}{0.96\linewidth}
\begin{minipage}{0.13\linewidth}
Step $2$
\end{minipage}%
\begin{minipage}[t]{0.87\linewidth}
Let $A$ be a cluster tilted algebra with corresponding quiver
\begin{center}
\includegraphics[scale=0.9]{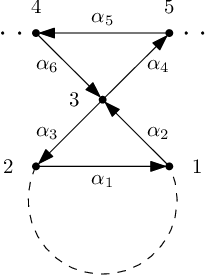}
\end{center}
\end{minipage}
\end{minipage}

\bigskip

\begin{minipage}{0.96\linewidth}
\begin{minipage}{0.13\linewidth}
\ 
\end{minipage}%
\begin{minipage}[c]{0.87\linewidth}
We define a tilting complex $T$ as follows:
Let $T=\bigoplus_{i=1}^{n+1} T_i$ be the following bounded complex of
projective $A$-modules,
where $T_i : 0 \rightarrow P_i \rightarrow 0, \ i \in \{1,2,4,\dots,n+1\}$,
are complexes concentrated in degree zero
and $T_3 : 0 \rightarrow P_3 \stackrel{(\alpha_2, \alpha_6)}{\longrightarrow}
P_1 \oplus P_4 \rightarrow 0$ is a complex concentrated in degrees $-1$ and $0$.
\end{minipage}
\end{minipage}

\bigskip

\begin{minipage}{0.96\linewidth}
\begin{minipage}{0.13\linewidth}
\ 
\end{minipage}%
\begin{minipage}[c]{0.87\linewidth}
By Rickard's theorem, $E:=\mathrm{End}_{D^b(A)}(T)$ is derived equivalent to $A$.
Using the alternating sum formula of the Proposition of Happel we can compute
the Cartan matrix of $E$ to be
{\scriptsize $\left(\begin{array}[c]{cccccc} 1&0&0&0&1&\dots\\ 1&1&1&0&0&\dots\\
1&0&1&1&0&\dots\\ 0&1&0&1&0&\dots\\ 0&0&1&1&1&\dots\\
\vdots&\vdots&\vdots&\vdots&\vdots&\ddots\end{array}\right)$}.
\end{minipage}
\end{minipage}

\begin{minipage}{0.96\linewidth}
\begin{minipage}{0.13\linewidth}
\ 
\end{minipage}%
\begin{minipage}[c]{0.87\linewidth}
(This deals with the case where not all the arrows between $2$ and $1$ along the
non-oriented cycle are oriented in the same direction. The case where they are can
be handled similarly.)
\end{minipage}
\end{minipage}

\bigskip

\begin{minipage}{0.96\linewidth}
\begin{minipage}{0.13\linewidth}
\ 
\end{minipage}%
\begin{minipage}[c]{0.87\linewidth}
We define homomorphisms in $E$ as follows
\begin{center}
\includegraphics[scale=0.9]{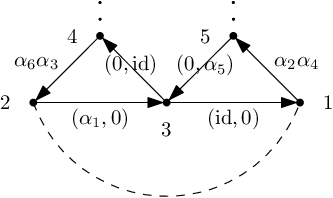}
\end{center}
\end{minipage}
\end{minipage}

\bigskip

\begin{minipage}{0.96\linewidth}
\begin{minipage}{0.13\linewidth}
\ 
\end{minipage}%
\begin{minipage}[c]{0.87\linewidth}
Thus, $A$ is derived equivalent to $E$ and $A^{\mathrm{op}}$ is derived
equivalent to $E^{\mathrm{op}}$, where the quiver of $E$ is the same as
the quiver we obtain after mutating at $3$.
\end{minipage}
\end{minipage}

\bigskip

In the Steps $3$ and $4$ of the proof of Lemma \ref{mutate_to_normalform} we mutate at a
vertex with three incident arrows as in Step $1$.
In Step $5$ we mutate at sinks, sources and at vertices with four incident arrows as in Step $2$.

Thus, we obtain a quiver of a derived equivalent cluster tilted algebra by all mutations
in the proof of Lemma \ref{mutate_to_normalform}.
Hence, every cluster tilted algebra $A=KQ/I$ of type $\tilde{A}_n$ is derived equivalent
to a cluster tilted algebra with a quiver in normal form which has the same parameters as $Q$.
\end{proof}

Our next aim is to prove the main result:

\begin{thm} \label{main_result}
Two cluster tilted algebras of type $\tAn$ are derived equivalent if and only if their
quivers have the same parameters $r_1,$ $r_2,$ $s_1$ and $s_2$ (up to changing the roles
of $r_i$ and $s_i$, $i\in \{1,2\}$).
\end{thm}

But first, we recall some background from \cite{Avella-Geiss}.
Let $A = KQ/I$ be a gentle algebra, where $Q=(Q_0,Q_1)$ is a connected quiver.
A {\em permitted path} of $A$ is a path $C=\alpha_l\dots\alpha_2\alpha_1$ which
contains no zero-relations.
A permitted path $C$ is called a {\em non-trivial permitted thread} if for all
$\beta \in Q_1$ neither $C \beta$ nor $\beta C$ is a permitted path.
Similarly a {\em forbidden path} of $A$ is a sequence $\Pi = \alpha_l\dots\alpha_2\alpha_1$
formed by pairwise different arrows in $Q$ with $\alpha_{i+1}\alpha_i \in I$ for all
$i \in \{1,2,\dots,l-1\}$.
A forbidden path $\Pi$ is called a {\em non-trivial forbidden thread} if for all
$\beta \in Q_1$ neither $\Pi \beta$ nor $\beta \Pi$ is a forbidden path.
Let $v \in Q_0$ such that $\# \{\alpha \in Q_1 : s(\alpha) = v\} \leq 1$,
$\# \{\alpha \in Q_1 : t(\alpha) = v\} \leq 1$
and if $\beta, \gamma \in Q_1$ are such that $s(\gamma) = v = t(\beta)$ then
$\gamma \beta \notin I$. Then we consider $e_v$ a {\em trivial permitted thread}
in $v$ and denote it by $h_v$. Let $\mH_A$ be the set of all permitted threads
of $A$, trivial and non-trivial. Similarly, for $v \in Q_0$ such that
$\# \{\alpha \in Q_1 : s(\alpha) = v\} \leq 1$, $\# \{\alpha \in Q_1 : t(\alpha) = v\} \leq 1$
and if $\beta, \gamma \in Q_1$ are such that $s(\gamma) = v = t(\beta)$ then
$\gamma \beta \in I$, we consider $e_v$ a {\em trivial forbidden thread} in $v$
and denote it by $p_v$. Note that certain paths can be permitted and forbidden
threads simultaneously.

Now, one can define two functions $\sigma, \varepsilon : Q_1 \rightarrow \{1,-1\}$ which
satisfy the following three conditions:
\begin{itemize}
\item[1)] If $\beta_1 \not= \beta_2$ are arrows with $s(\beta_1) = s(\beta_2)$,
then $\sigma(\beta_1) = -\sigma(\beta_2)$.
\item[2)] If $\gamma_1 \not= \gamma_2$ are arrows with $t(\gamma_1) = t(\gamma_2)$,
then $\varepsilon(\gamma_1) = -\varepsilon(\gamma_2)$.
\item[3)] If $\beta$ and $\gamma$ are arrows with $s(\gamma) = t(\beta)$ and
$\gamma \beta \notin I$, then $\sigma(\gamma) = -\varepsilon(\beta)$.
\end{itemize}
We can extend these functions to threads of $A$ as follows:
for a non-trivial thread $H=\alpha_l\dots\alpha_2\alpha_1$ of $A$ define
$\sigma(H):= \sigma(\alpha_1)$ and $\varepsilon(H):= \varepsilon(\alpha_l)$.
If there is a trivial permitted thread $h_v$ for some $v \in Q_0$, the connectivity
of $Q$ assures the existence of some $\gamma \in Q_1$ with $s(\gamma) = v$ or some
$\beta \in Q_1$ with $t(\beta) = v$.
In the first case, we define $\sigma(h_v) = -\varepsilon(h_v) := -\sigma(\gamma)$,
for the second case $\sigma(h_v) = -\varepsilon(h_v) := \varepsilon(\beta)$.
If there is a trivial forbidden thread $p_v$ for some $v \in Q_0$, we know that there
exists $\gamma \in Q_1$ with $s(\gamma) = v$ or $\beta \in Q_1$ with $t(\beta) = v$.
In the first case, we define $\sigma(p_v) = \varepsilon(h_v) := -\sigma(\gamma)$,
for the second case $\sigma(p_v) = \varepsilon(h_v) := -\varepsilon(\beta)$.

Now there is a combinatorial algorithm (stated in \cite{Avella-Geiss}) to produce
certain pairs of natural numbers, by using only the quiver with relations which
defines a gentle algebra.
In the algorithm we are going forward through permitted threads and backwards through
forbidden threads in such a way that each arrow and its inverse is used exactly once.

\begin{alg} \label{algorithm_avella}
The algorithm is as follows:
\begin{itemize}
\item[1)]
\begin{itemize}
\item[a)] Begin with a permitted thread $H_0$ of $A$.
\item[b)] If $H_i$ is defined, consider $\Pi_i$ the forbidden thread which
ends in $t(H_i)$ and such that $\varepsilon(H_i) = -\varepsilon(\Pi_i)$.
\item[c)] Let $H_{i+1}$ be the permitted thread which starts in $s(\Pi_i)$
and such that $\sigma(H_{i+1}) = -\sigma(\Pi_i)$.

The process stops when $H_k = H_0$ for some natural number $k$.
Let $m = \sum\limits_{1 \leq i \leq k} l(\Pi_{i-1})$,
where $l()$ is the length of a path, i.e., the number of arrows of the path.
We obtain the pair $(k,m)$.
\end{itemize}
\item[2)] Repeat the first step of the algorithm until all permitted threads of
$A$ have been considered.
\item[3)] If there are oriented cycles in which each pair of consecutive arrows
form a relation, we add a pair $(0,m)$ for each of those cycles, where $m$ is the
length of the cycle.
\item[4)] Define $\phi_A : \N^2 \rightarrow \N$, where $\phi_A(k,m)$ is the number
of times the pair $(k,m)$ arises in the algorithm.
\end{itemize}
\end{alg}

This function $\phi$ is invariant under derived equivalence:

\begin{lem}
[Avella-Alaminos and Gei\ss \ \cite{Avella-Geiss}] \label{phi_invariant}
Let $A$ and $B$ be gentle algebras. If $A$ and $B$ are derived equivalent,
then $\phi_A = \phi_B$.
\end{lem}

\begin{figure}[t!]
\centering
\includegraphics[scale=0.75]{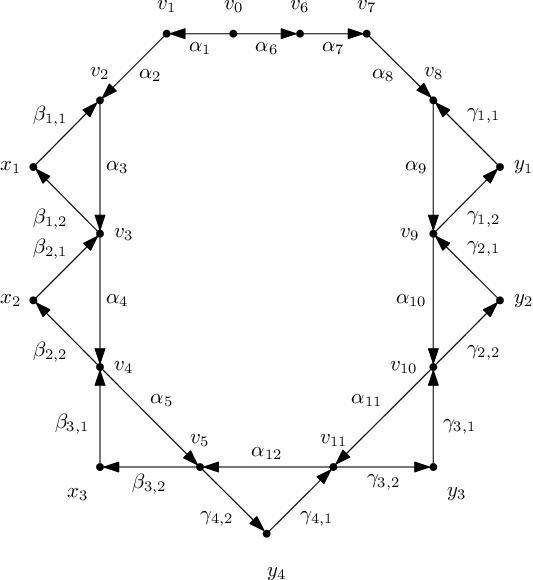}
\caption{Quiver for Example~\ref{ex:algorithm}.}
\label{fig:example_algorithm}
\end{figure}

\begin{eg}\label{ex:algorithm}
Figure~\ref{fig:example_algorithm} shows the quiver of a cluster tilted algebra $A$
of type $\tilde{A}_{18}$, where $r_1=2$, $r_2=3$, $s_1=3$ and $s_2=4$ and thus, $r:=r_1+r_2=5$
and $s:=s_1+s_2=7$.



Now, we define the functions $\sigma$ and $\varepsilon$ for all arrows in $Q$:
\begin{longtable}[c]{lcrlcrlccl}
$\sigma(\alpha_i)$ &$=$& $1$, & $\varepsilon(\alpha_i)$ &$=$& $-1$ & for all
& $i$ &$=$& $1, \dots, 5$\\
$\sigma(\alpha_i)$ &$=$& $-1$, & $\varepsilon(\alpha_i)$ &$=$& $1$ & for all
& $i$ &$=$& $6, \dots, 12$\\
$\sigma(\beta_{j,1})$ &$=$& $1$, & $\varepsilon(\beta_{j,1})$ &$=$& $1$ & for all
& $j$ &$=$& $1, \dots, 3$\\
$\sigma(\beta_{j,2})$ &$=$& $-1$, & $\varepsilon(\beta_{j,2})$ &$=$& $1$ & for all
& $j$ &$=$& $1, \dots, 3$\\
$\sigma(\gamma_{l,1})$ &$=$& $-1$, & $\varepsilon(\gamma_{l,1})$ &$=$& $-1$ & for all
& $l$ &$=$& $1, \dots, 4$\\
$\sigma(\gamma_{l,2})$ &$=$& $1$, & $\varepsilon(\gamma_{l,2})$ &$=$& $-1$ & for all
& $l$ &$=$& $1, \dots, 4$\\
\end{longtable}
Then $\mH_A$ is formed by $h_{v_1},$ $h_{v_6},$ $h_{v_7},$
$\gamma_{4,2}\alpha_5\alpha_4\alpha_3\alpha_2\alpha_1,$ $\beta_{3,2}\alpha_{12}\alpha_{11}\alpha_{10}\alpha_9\alpha_8\alpha_7\alpha_6,$
$\beta_{1,1},$ $\beta_{1,2}\beta_{2,1},$ $\beta_{2,2}\beta_{3,1},$ $\gamma_{1,1},$
$\gamma_{1,2}\gamma_{2,1},$ $\gamma_{2,2}\gamma_{3,1}$ and $\gamma_{3,2}\gamma_{4,1}$.
The forbidden threads of $A$ are $p_{x_1},$ $p_{x_2},$ $p_{x_3},$ $p_{y_1},$
$p_{y_2},$ $p_{y_3},$ $p_{y_4},$ $\alpha_1,$ $\alpha_2,$ $\alpha_6,$ $\alpha_7,$
$\alpha_8$ and all the oriented $3$-cycles.

\medskip

Moreover, we can write
\begin{center}
$\begin{array}{lclclclcr}
\sigma(h_{v_1}) &=& -\varepsilon(h_{v_1}) &=& -\sigma(\alpha_2)
&=& \varepsilon(\alpha_1) &=& -1\\
\sigma(h_{v_6}) &=& -\varepsilon(h_{v_6}) &=& -\sigma(\alpha_7)
&=& \varepsilon(\alpha_6) &=& 1\\
\sigma(h_{v_7}) &=& -\varepsilon(h_{v_7}) &=& -\sigma(\alpha_8)
&=& \varepsilon(\alpha_7) &=& 1
\end{array}$
\end{center}
for the trivial permitted threads and
\begin{center}
$\begin{array}{lclclclcrlccl}
\sigma(p_{x_i}) &=& \varepsilon(p_{x_i}) &=& -\sigma(\beta_{i,1})
&=& -\varepsilon(\beta_{i,2}) &=& -1 & \textrm{ for all } & i &=& 1,2,3\\
\sigma(p_{y_i}) &=& \varepsilon(p_{y_i}) &=& -\sigma(\gamma_{i,1})
&=& -\varepsilon(\gamma_{i,2}) &=& 1 & \textrm{ for all } & i &=& 1,2,3,4
\end{array}$
\end{center}
for the trivial forbidden threads.

\medskip

Let $H_0=h_{v_1}$ and $\Pi_0 = \alpha_1$ with $\varepsilon(h_{v_1})=
-\varepsilon(\alpha_1) =1$.
Then $H_1$ is the permitted thread which starts in $s(\Pi_0) = v_0$
and $\sigma(H_1) = \sigma(\alpha_6) = -\sigma(\Pi_0) =-1$, that is
$\beta_{3,2}\alpha_{12}\alpha_{11}\alpha_{10}\alpha_9\alpha_8\alpha_7\alpha_6$.
Now $\Pi_1 = p_{x_3}$ since it is the forbidden thread which ends in $x_3$
and $\varepsilon(\Pi_1) = -\varepsilon(H_1) = -\varepsilon(\beta_{3,2}) =-1$.
Then $H_2 = \beta_{2,2}\beta_{3,1}$ is the permitted thread starting in $x_3$ and
$\sigma(\Pi_1) = -\sigma(H_2) = -\sigma(\beta_{3,1}) =-1$.
Thus, $\Pi_2 = p_{x_2}$ with $\varepsilon(H_2) = \varepsilon(\beta_{2,2}) =
-\varepsilon(\Pi_2) =1$.

In the same way we can define the missing threads and we get:
\begin{longtable}{lcllcl}
$H_0$ & = & $h_{v_1}$ & $\Pi_0^{-1}$ &=& $\alpha_1^{-1}$\\
$H_1$ & = & $\beta_{3,2}\alpha_{12}\alpha_{11}\alpha_{10}\alpha_9\alpha_8\alpha_7\alpha_6$
& $\Pi_1^{-1}$ &=& $p_{x_3}$\\
$H_2$ & = & $\beta_{2,2}\beta_{3,1}$ & $\Pi_2^{-1}$ &=& $p_{x_2}$\\
$H_3$ & = & $\beta_{1,2}\beta_{2,1}$ & $\Pi_3^{-1}$ &=& $p_{x_1}$\\
$H_4$ & = & $\beta_{1,1}$ & $\Pi_4^{-1}$ &=& $\alpha_2^{-1}$\\
$H_5$ & = & $H_0$
\end{longtable}
\begin{center}
$\rightarrow (5,2)$
\end{center}
where $\alpha_1^{-1}$ is defined by $s(\alpha_1^{-1}) := t(\alpha_1)$, $t(\alpha_1^{-1})
:= s(\alpha_1)$ and $(\alpha_1^{-1})^{-1} = \alpha_1$.

\newpage

If we continue with the algorithm we obtain the second pair $(7,3) = (s,s_1)$ in the
following way:
\begin{longtable}{lcllcl}
$H_0$ & = & $h_{v_6}$ & $\Pi_0^{-1}$ &=& $\alpha_6^{-1}$\\
$H_1$ & = & $\gamma_{4,2}\alpha_5\alpha_4\alpha_3\alpha_2\alpha_1$ & $\Pi_1^{-1}$ &=& $p_{y_4}$\\
$H_2$ & = & $\gamma_{3,2}\gamma_{4,1}$ & $\Pi_2^{-1}$ &=& $p_{y_3}$\\
$H_3$ & = & $\gamma_{2,2}\gamma_{3,1}$ & $\Pi_3^{-1}$ &=& $p_{y_2}$\\
$H_4$ & = & $\gamma_{1,2}\gamma_{2,1}$ & $\Pi_4^{-1}$ &=& $p_{y_1}$\\
$H_5$ & = & $\gamma_{1,1}$ & $\Pi_5^{-1}$ &=& $\alpha_8^{-1}$\\
$H_6$ & = & $h_{v_7}$ & $\Pi_6^{-1}$ &=& $\alpha_7^{-1}$\\
$H_7$ & = & $H_0$
\end{longtable}
\begin{center}
$\rightarrow (7,3)$
\end{center}
Finally, we have to add seven pairs $(0,3)$ for the seven oriented $3$-cycles.
Thus, we get $\phi_A(5,2) = 1, \ \phi_A(7,3) =1$ and $\phi_A(0,3) =7$.
\end{eg}

Now we can extend this example to general quivers of cluster tilted algebras
of type $\tilde{A}_n$ in normal form.

\begin{proof}[Proof of Theorem \ref{main_result}]
We know from Lemma \ref{derived_equivalences} that every cluster tilted algebra
$A=KQ/I$ of type $\tAn$ with parameters
$r_1, \ r_2, \ s_1$ and $s_2$ is derived equivalent to a cluster tilted algebra
with a quiver in normal form, as shown in Figure~\ref{normal_form},
where $r_1$ is the number of arrows in the anti-clockwise direction which do not
share any arrow with an oriented $3$-cycle and
$s_1$ is the number of arrows in the clockwise direction which do not share any
arrow with an oriented $3$-cycle.
Moreover, $r_2$ is the number of oriented $3$-cycles which share one arrow $\alpha$
with the non-oriented cycle and $\alpha$ is oriented in the anti-clockwise direction and
$s_2$ is the number of oriented $3$-cycles which share one arrow $\beta$ with the
non-oriented cycle and $\beta$ is oriented in the clockwise direction
(see Definition \ref{parameters}).
Thus, $r:=r_1 + r_2$ is the number of arrows of the non-oriented cycle
in the anti-clockwise direction and $s:=s_1 + s_2$ is the number of arrows of
the non-oriented cycle in the clockwise direction.

\smallskip

We consider the quiver $Q$ in normal form with notation as given in
Figure~\ref{normal_form_avella} and
define the functions $\sigma$ and $\varepsilon$ for all arrows in $Q$:
\begin{center}
$\begin{array}[c]{lcrlcrlccl}
\sigma(\alpha_i) &=& 1, & \varepsilon(\alpha_i) &=& -1 & \textrm{ for all }
& i &=& 1, \dots, r\\
\sigma(\alpha_i) &=& -1, & \varepsilon(\alpha_i) &=& 1 & \textrm{ for all }
& i &=& r+1, \dots, r+s\\
\sigma(\beta_{j,1}) &=& 1, & \varepsilon(\beta_{j,1}) &=& 1 & \textrm{ for all }
& j &=& 1, \dots, r_2\\
\sigma(\beta_{j,2}) &=& -1, & \varepsilon(\beta_{j,2}) &=& 1 & \textrm{ for all }
& j &=& 1, \dots, r_2\\
\sigma(\gamma_{l,1}) &=& -1, & \varepsilon(\gamma_{l,1}) &=& -1 & \textrm{ for all }
& l &=& 1, \dots, s_2\\
\sigma(\gamma_{l,2}) &=& 1, & \varepsilon(\gamma_{l,2}) &=& -1 & \textrm{ for all }
& l &=& 1, \dots, s_2\\
\end{array}$
\end{center}
Here $\mH_A$ is formed by $h_{v_1}, \dots, h_{v_{r_1-1}}$,
$h_{v_{r+1}},\dots, h_{v_{r+s_1-1}},$
$\beta_{r_2,2}\alpha_{r+s}\alpha_{r+s-1}\dots\alpha_{r+2}\alpha_{r+1},$

\noindent
$\gamma_{s_2,2}\alpha_r\alpha_{r-1}\dots\alpha_2\alpha_1$, $\beta_{1,1}$,
$\beta_{1,2}\beta_{2,1}, \dots, \beta_{r_2-1,2}\beta_{r_2,1},$ $\gamma_{1,1}$,
$\gamma_{1,2}\gamma_{2,1},\dots, \gamma_{s_2-1,2}\gamma_{s_2,1}$.

The forbidden threads of $A$ are $p_{x_1},\dots, p_{x_{r_2}},$ $p_{y_1},\dots, p_{y_{s_2}},$
$\alpha_1, \dots, \alpha_{r_1},$ $\alpha_{r+1},\dots, \alpha_{r+s_1}$ and all the
oriented $3$-cycles.

\medskip

Moreover, we can write
\begin{longtable}[c]{lclclclcr}
$\sigma(h_{v_1})$ &$=$& $-\varepsilon(h_{v_1})$ &$=$& $-\sigma(\alpha_2)$
&$=$& $\varepsilon(\alpha_1)$ &$=$& $-1$\\
& $\vdots$\\
$\sigma(h_{v_{r_1-1}})$ &$=$& $-\varepsilon(h_{v_{r_1-1}})$ &$=$& $-\sigma(\alpha_{r_1})$
&$=$& $\varepsilon(\alpha_{r_1-1})$ &$=$& $-1$\\
$\sigma(h_{v_{r+1}})$ &$=$& $-\varepsilon(h_{v_{r+1}})$ &$=$& $-\sigma(\alpha_{r+2})$
&$=$& $\varepsilon(\alpha_{r+1})$ &=& $1$\\
& $\vdots$\\
$\sigma(h_{v_{r+s_1-1}})$ &$=$& $-\varepsilon(h_{v_{r+s_1-1}})$ &$=$& $-\sigma(\alpha_{r+s_1})$
&$=$& $\varepsilon(\alpha_{r+s_1-1})$ &$=$& $1$
\end{longtable}
for the trivial permitted threads and
\begin{center}
$\begin{array}[c]{lclclclcrlccl}
\sigma(p_{x_i}) &=& \varepsilon(p_{x_i}) &=& -\sigma(\beta_{i,1})
&=& -\varepsilon(\beta_{i,2}) &=& -1 & \textrm{ for all } & i &=& 1, \dots, r_2\\
\sigma(p_{y_i}) &=& \varepsilon(p_{y_i}) &=& -\sigma(\gamma_{i,1})
&=& -\varepsilon(\gamma_{i,2}) &=& 1 & \textrm{ for all } & i &=& 1, \dots, s_2
\end{array}$
\end{center}
for the trivial forbidden threads.

\medskip

Thus, we can apply the Algorithm~\ref{algorithm_avella} as follows:
\begin{longtable}[c]{lcllcl}
$H_0$ & = & $h_{v_1}$ & $\Pi_0^{-1}$ &=& $\alpha_1^{-1}$\\
$H_1$ & = & $\beta_{r_2,2}\alpha_{r+s}\alpha_{r+s-1}\dots\alpha_{r+2}\alpha_{r+1}$
& $\Pi_1^{-1}$ &=& $p_{x_{r_2}}$\\
$H_2$ & = & $\beta_{r_2-1,2}\beta_{r_2,1}$ & $\Pi_2^{-1}$ &=& $p_{x_{r_2-1}}$\\
\vdots & & & \vdots\\
$H_{r_2}$ & = & $\beta_{1,2}\beta_{2,1}$ & $\Pi_{r_2}^{-1}$ &=& $p_{x_1}$\\
$H_{r_2+1}$ & = & $\beta_{1,1}$ & $\Pi_{r_2+1}^{-1}$ &=& $\alpha_{r_1}^{-1}$\\
$H_{r_2+2}$ & = & $h_{v_{r_1-1}}$ & $\Pi_{r_2+2}^{-1}$ &=& $\alpha_{r_1-1}^{-1}$\\
\vdots & & & \vdots\\
$H_{r-1}$ & = & $h_{v_2}$ & $\Pi_{r-1}^{-1}$ &=& $\alpha_2^{-1}$\\
$H_r$ & = & $H_0$
\end{longtable}

\begin{center}
$\begin{array}[c]{lcl}
m & = & l(\Pi_0) + l(\Pi_{r_2+1}) + l(\Pi_{r_2+2}) + \dots + l(\Pi_{r-1})\\
 & = & 1 + \underbrace{1+ 1 + \dots + 1}_{((r-1)-r_2)-\textrm{times}}\\
 & = & 1+(r-1)-r_2\\
 & = & r-r_2\\
 & = & r_1
\end{array}$

$\rightarrow (r,r_1)$
\end{center}
If we continue with the algorithm we obtain the second pair $(s,s_1)$
in the following way:
\begin{longtable}[c]{lcllcl}
$H_0$ & = & $h_{v_{r+1}}$ & $\Pi_0^{-1}$ &=& $\alpha_{r+1}^{-1}$\\
$H_1$ & = & $\gamma_{s_2,2}\alpha_r\alpha_{r-1}\dots\alpha_2\alpha_1 $ &
$\Pi_1^{-1}$ &=& $p_{y_{s_2}}$\\
$H_2$ & = & $\gamma_{s_2-1,2}\gamma_{s_2,1}$ & $\Pi_2^{-1}$ &=& $p_{y_{s_2-1}}$\\
\vdots & & & \vdots\\
$H_{s_2}$ & = & $\gamma_{1,2}\gamma_{2,1}$ & $\Pi_{s_2}^{-1}$ &=& $p_{y_1}$\\
$H_{s_2+1}$ & = & $\gamma_{1,1}$ & $\Pi_{s_2+1}^{-1}$ &=& $\alpha_{r+s_1}^{-1}$\\
$H_{s_2+2}$ & = & $h_{v_{r+s_1-1}}$ & $\Pi_{s_2+2}^{-1}$ &=& $\alpha_{r+s_1-1}^{-1}$\\
\vdots & & & \vdots\\
$H_{s-1}$ & = & $h_{v_{r+2}}$ & $\Pi_{s-1}^{-1}$ &=& $\alpha_{r+2}^{-1}$\\
$H_s$ & = & $H_0$
\end{longtable}

\begin{center}
$\rightarrow (s,s_1)$
\end{center}
Finally, we have to add $r_2+s_2$ pairs $(0,3)$ for the oriented $3$-cycles.
Thus, we have $\phi_A(r,r_1) = 1,$ $\phi_A(s,s_1) =1$ and $\phi_A(0,3) = r_2+s_2$,
where $r=r_1+r_2$ and $s=s_1+s_2$.

\medskip

Now, let $A$ and $B$ be two cluster tilted algebras of type $\tAn$ with parameters
$r_1, \ r_2, \ s_1, \ s_2$, respectively
$\tilde{r}_1, \ \tilde{r}_2, \ \tilde{s}_1, \ \tilde{s}_2$.
From above we can conclude that $\phi_A=\phi_B$ if and only if $r_1 = \tilde{r}_1,$
$r_2 = \tilde{r}_2,$ $s_1 = \tilde{s}_1$ and $s_2 = \tilde{s}_2$ or
$r_1 = \tilde{s}_1,$ $r_2 = \tilde{s}_2,$ $s_1 = \tilde{r}_1$ and $s_2 = \tilde{r}_2$
(which ends up being the same quiver).

Hence, if $A$ is derived equivalent to $B$, we know from Theorem \ref{phi_invariant}
that $\phi_A=\phi_B$ and thus, that the parameters are the same.
Otherwise, if $A$ and $B$ have the same parameters, they are both derived equivalent
to the same cluster tilted algebra with a quiver in normal form.
\end{proof}

\begin{figure}
\centering
\includegraphics[scale=0.75]{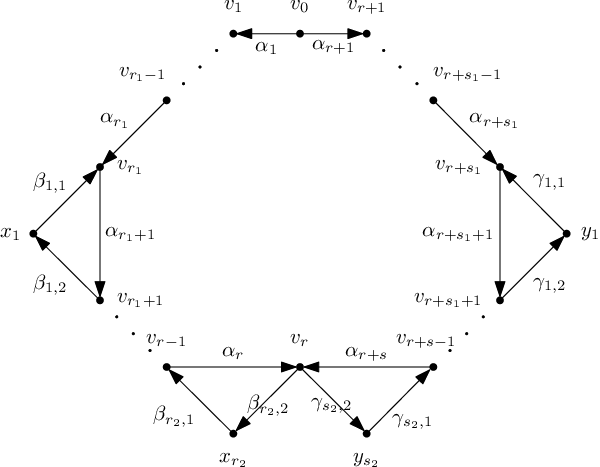}
\caption{A quiver in normal form.}
\label{normal_form_avella}
\end{figure}


\end{document}